\documentclass[11pt]{article}

\usepackage{amssymb,amsmath,amsfonts,amsthm}
\usepackage{mathtools}
\usepackage{latexsym}
\usepackage{graphics}
\usepackage{indentfirst}
\usepackage{enumerate}
\usepackage{url}

\newtheorem{innercustomthm}{Theorem}
\newenvironment{customthm}[1]
  {\renewcommand\theinnercustomthm{#1}\innercustomthm}
  {\endinnercustomthm}

\newtheorem*{thm*}{Theorem}
\newtheorem{thm}{Theorem}
\newtheorem{lem}[thm]{Lemma}

\newtheorem{obs}[thm]{Observation}

\newtheorem{ques}[thm]{Question}
\newtheorem{defin}[thm]{Definition}

\newcommand{\N}{\mathbb{N}}
\newcommand{\Z}{\mathbb{Z}}

\newcommand{\R}{\mathbb{R}}

\newcommand{\col}{\mathrm{col}}
\newcommand{\card}[1]{\left \vert #1 \right \vert}

\begin{document}

\title{List Coloring the Cartesian Product of a Complete Graph and Complete Bipartite Graph}

\author{Hemanshu Kaul\footnote{Department of Applied Mathematics, Illinois Institute of Technology, Chicago, IL 60616. E-mail: {\tt kaul@illinoistech.edu}} \\
Leonardo Marciaga\footnote{Department of Applied Mathematics, Illinois Institute of Technology, Chicago, IL 60616. E-mail: {\tt lmarciaga@hawk.illinoistech.edu}} \\
Jeffrey A. Mudrock\footnote{Department of Mathematics and Statistics, University of South Alabama, Mobile, AL 36688. E-mail: {\tt mudrock@southalabama.edu}} }

\maketitle

\begin{abstract}
We study the list chromatic number of the Cartesian product of a complete graph of order $n$ and a complete bipartite graph with partite sets of size $a$ and $b$, denoted $\chi_{\ell}(K_n \square K_{a,b})$.  At the 2024 Sparse Graphs Coalition's Workshop on algebraic, extremal, and structural methods and problems in graph colouring, Mudrock presented the following question: For each positive integer $a$, does $\chi_{\ell}(K_n \square K_{a,b}) = n+a$ if and only if $b \geq (n+a-1)!^a/(a-1)!^a$?  In this paper, we show the answer to this question is yes by studying $\chi_{\ell}(H \square K_{a,b})$ when $H$ is strongly chromatic-choosable (a special form of vertex criticality) with the help of the list color function and analytic inequalities such as that of Karamata. Our result can be viewed as a generalization of the well-known result that $\chi_{\ell}(K_{a,b}) = 1+a$ if and only if $b \geq a^a$.
\medskip



\end{abstract}

\section{Introduction}\label{intro}

In this paper all graphs are nonempty, finite, simple graphs unless otherwise noted.  Generally speaking we follow West~\cite{W01} for terminology and notation.  The set of natural numbers is $\N = \{1,2,3, \dots \}$.  For $k \in \N$, we write $[k]$ for the set $\{1, \dots, k \}$ and $[0] = \emptyset$. We adopt the convention that $\prod_{i=a}^b x_i = 1$ whenever $a > b$. We use AM-GM inequality to mean the Inequality of Arithmetic and Geometric Means. If $G$ is a graph and $S \subseteq V(G)$, we write $G[S]$ for the subgraph of $G$ induced by $S$.  For $v \in V(G)$, we write $d_G(v)$ for the degree of vertex $v$ in the graph $G$, and we write $N_G(v)$ for the neighborhood of $v$ in $G$.  If $G$ and $H$ are vertex disjoint graphs, the \emph{join} of $G$ and $H$, denoted $G \vee H$, is the graph consisting of $G$, $H$, and additional edges added so that each vertex in $G$ is adjacent to each vertex in $H$.

\subsection{List Coloring Cartesian Products}

List coloring is a variation on the classical vertex coloring problem that was introduced in the 1970s independently by Vizing~\cite{V76} and Erd\H{o}s, Rubin, and Taylor~\cite{ET79}.  In the classical vertex coloring problem we seek a \emph{proper $k$-coloring} of a graph $G$ which is a coloring of the vertices of $G$ with colors from $[k]$ so that adjacent vertices receive different colors. The \emph{chromatic number} of a graph, denoted $\chi(G)$, is the smallest $k$ such that $G$ has a proper $k$-coloring.  For list coloring, we associate a \emph{list assignment} $L$ with a graph $G$ which assigns to each vertex $v \in V(G)$ a list of colors $L(v)$ (we say $L$ is a list assignment for $G$).  The graph $G$ is \emph{$L$-colorable} if there exists a proper coloring $f$ of $G$ such that $f(v) \in L(v)$ for each $v \in V(G)$ (we refer to $f$ as a \emph{proper $L$-coloring} of $G$).  A list assignment $L$ is called a \emph{k-assignment} for $G$ if $|L(v)|=k$ for each $v \in V(G)$.  The \emph{list chromatic number} of a graph $G$, denoted $\chi_\ell(G)$, is the smallest $k$ such that $G$ is $L$-colorable whenever $L$ is a $k$-assignment for $G$.  We say $G$ is \emph{$k$-choosable} if $k \geq \chi_\ell(G)$.

The \emph{Cartesian product} of graphs $M$ and $H$, denoted $M \square H$, is the graph with vertex set $V(M) \times V(H)$ and edges created so that $(u,v)$ is adjacent to $(u',v')$ if and only if either $u=u'$ and $vv' \in E(H)$ or $v=v'$ and $uu' \in E(M)$.  Throughout this paper, if $G = M \square H$ and $u \in V(M)$ (resp. $u \in V(H)$), we let $V_u$ be the subset of $V(G)$ consisting of the vertices with first (resp. second) coordinate $u$.  We also let $G_u = G[V_u]$.  Similarly, if $S \subseteq V(M) \cup V(H)$, we let $V_S = \bigcup_{s \in S} V_s$ and $G_S = G[V_S]$.  By the definition of the Cartesian product of graphs, it is easy to see that $G_u$ is a copy of $H$ (resp. $M$) when $u \in V(M)$ (resp. $u \in V(H)$).  When $L$ is a list assignment for $G$ and $u \in V(G)$, we write $L_u$ for the list assignment for $G_u$ obtained by restricting the domain of of $L$ to $V_u$.  Similarly, when $S \subseteq V(M) \cup V(H)$, we write $L_S$ for the list assignment for $G_S$ obtained by restricting the domain of $L$ to $V_S$.

It is well-known that $\chi(G \square H) = \max \{\chi(G), \chi(H) \}$.  On the other hand, the list chromatic number of the Cartesian product of graphs is not nearly as well understood.  In 2006, Borowiecki, Jendrol, Kr{\'a}l, and Mi{\v s}kuf~\cite{BJ06} showed the following.

\begin{thm}[\cite{BJ06}] \label{thm: Borow1}
For any graphs $G$ and $H$, $\chi_\ell(G \square H) \leq \min \{\chi_\ell(G) + \col(H), \col(G) + \chi_\ell(H) \} - 1.$
\end{thm}

Here $\col(G)$ denotes the \emph{coloring number} of a graph $G$ which is the smallest integer $d$ for which there exists an ordering, $v_1, \dots, v_n$, of the elements in $V(G)$ such that each vertex $v_i$ has at most $d-1$ neighbors among $v_1, \dots, v_{i-1}$.  For this paper, it is important to note that Theorem~\ref{thm: Borow1} implies $\chi_\ell(G \square K_{a,b}) \leq \chi_\ell(G) + a$.

It is also proven in~\cite{BJ06} that the bound in Theorem~\ref{thm: Borow1} is tight.   

\begin{thm}[\cite{BJ06}] \label{thm: Borow2}
Suppose $G$ is a graph with $n$ vertices.  Then, $\chi_\ell(G \square K_{a,b}) = \chi_\ell(G) + a$ whenever $b \geq (\chi_\ell(G) + a - 1)^{an}$.
\end{thm}

It is natural to wonder when the bound on $b$ in Theorem~\ref{thm: Borow2} is best possible.  With this in mind, for each $a \in \N$, we let $f_a(G)$ be the smallest $b$ such that $\chi_\ell(G \square K_{a,b}) = \chi_\ell(G) + a$.  Note $\chi_\ell(G \square K_{a,0}) = \chi_\ell(G) < \chi_\ell(G) + a$ which implies that $f_a(G) \geq 1$.  Second, Theorem~\ref{thm: Borow2} implies that $f_a(G) \leq (\chi_\ell(G) + a - 1)^{a|V(G)|}$.  This means $f_a(G)$ exists and is a natural number.  Also, if $G$ is a disconnected graph with components: $H_1, H_2, \dots, H_r$, we have $f_a(G) = \max \{f_a(H_i) \colon i \in [r] ,\chi_{\ell}(H_i) = \chi_{\ell}(G)\}$.  So, we will restrict our attention to connected graphs from this point forward.

\subsection{The List Color Function and Strong Chromatic-Choosability}

Let $P(G,k)$ be the \emph{chromatic polynomial} of the graph $G$; that is, $P(G,k)$ is equal to the number of proper $k$-colorings of $G$.  It is known that $P(G,k)$ is a polynomial in $k$ (see~\cite{B12}).  In 1990~\cite{AS90} this notion was extended to list coloring as follows. If $L$ is a list assignment for $G$, we use $P(G,L)$ to denote the number of proper $L$-colorings of $G$. The \emph{list color function} $P_\ell(G,k)$ is the minimum value of $P(G,L)$ where the minimum is taken over all possible $k$-assignments $L$ for $G$.  Since a $k$-assignment could assign the same $k$ colors to every vertex in a graph, it is clear that $P_\ell(G,k) \leq P(G,k)$ for each $k \in \N$.  In general, the list color function of a graph can differ significantly from its chromatic polynomial for small values of $k$.  However, for large values of $k$, Dong and Zhang~\cite{DZ22} (improving upon results in~\cite{D92}, \cite{T09}, and~\cite{WQ17}) showed that for any graph $G$ with at least 2 edges, $P_{\ell}(G,k)=P(G,k)$ whenever $k \geq |E(G)|-1$.

In the case $G$ is a complete graph or a cycle, it is well known (see~\cite{R68}) that $P(C_{n},k)=(k-1)^{n}+(-1)^n(k-1)$ and $P(K_n,k) = \prod_{i=0}^{n-1} (k-i)$.  It is easy to see that for each $n,k \in \N$, $P(K_n,k)=P_{\ell}(K_n,k)$, and it was shown in~\cite{KN16} that for each $n,k \in \N$, $P(C_n,k) = P_{\ell}(C_n,k)$.

A graph is \emph{$k$-vertex critical} if its chromatic number is $k$ and the removal of any vertex in the graph decreases the chromatic number of the graph.  In~\cite{KM18} the first and third named authors introduced the related notion of strong chromatic-choosability and used the list color function to compute $f_a$ with $a=1$ for graphs that are strongly chromatic-choosable (Theorem~\ref{thm: star} below).  A graph $G$ is \emph{strongly $k$-chromatic-choosable} if it is $k$-vertex critical and every $(k-1)$-assignment $L$ for which $G$ is not $L$-colorable has the property that the lists are the same on all vertices.  List assignments that assign the same list of colors to every vertex of a graph are called \emph{constant}.  We say $G$ is \emph{strongly chromatic-choosable} if it is strongly $\chi(G)$-chromatic-choosable. Note that if $G$ is strongly $k$-chromatic-choosable, then the only reason $G$ is not $(k-1)$-choosable is that a proper $(k-1)$-coloring of $G$ does not exist. Simple examples of strongly chromatic-choosable graphs include complete graphs, odd cycles, and the join of a complete graph and odd cycle (see~\cite{BK24} and~\cite{KM18} for many other examples).

\begin{thm} [\cite{KM18}] \label{thm: star}
Let $M$ be a strongly $k$-chromatic-choosable graph.  Then, $f_1(M) = P_\ell(M,k)$.
\end{thm}

\subsection{Motivating Question}

The following general upper bound on $f_a(G)$ was proven in~\cite{KM19}.

\begin{thm} [\cite{KM19}] \label{thm: generalupper}
For any graph $G$ and $a \in \N$, $f_a(G) \leq (P_\ell(G, \chi_\ell(G) + a - 1))^a$.
\end{thm}

Notice Theorem~\ref{thm: star} shows the bound in Theorem~\ref{thm: generalupper} is tight when $a=1$ and $G$ is strongly chromatic-choosable.  However, it is not the case that $f_a(G) = (P_\ell(G, \chi_\ell(G) + a - 1))^a$ for all graphs $G$ and $a \in \N$ since it is easy to see that $f_1(C_{2n+2})=1$, yet $P_\ell(C_{2n+2}, 2)=2$.  This observation leads to the following open question.

\begin{ques} [\cite{KM19}] \label{ques: uppertight}
For what graphs does $f_a(G) = (P_\ell(G, \chi_\ell(G) + a - 1))^a$ for each $a \in \N$?
\end{ques}

In~\cite{KM19} some partial progress was made on Question~\ref{ques: uppertight} specifically in the case where our attention is restricted to strongly chromatic-choosable graphs.

\begin{thm} [\cite{KM19}] \label{thm: sccexact}
If $M$ is strongly chromatic-choosable and $\chi(M) \geq a + 1$, then $f_a(M) = (P_\ell(M, \chi_\ell(M) + a - 1))^a.$
\end{thm}

It is unknown whether there are any strongly chromatic-choosable graphs $M$ for which $f_a(M) < (P_\ell(M, \chi_\ell(M) + a - 1))^a.$  Consequently, the following question, which was presented by the third named author at the 2024 Sparse Graphs Coalition's Workshop on algebraic, extremal, and structural methods and problems in graph colouring~\cite{M24}, is open and served as the main motivation for this paper.

\begin{ques} [\cite{KM19, M24}] \label{ques: complete}
Is it the case that $f_a(K_n) = (P_{\ell}(K_n,n+a-1))^a = \left( \frac{(n+a-1)!}{(a-1)!} \right)^a$ for each $n, a \in \N$?
\end{ques}

In what follows we show that the answer to this question is yes.  In Subsection~\ref{reduce} we present several important lemmas and observations that are used in the proof of our main result.  Importantly, the results in Subsection~\ref{reduce} apply to all strongly chromatic-choosable graphs; so, they may be of independent interest since they could be used to explore whether all strongly chromatic-choosable graphs satisfy the condition in Question~\ref{ques: uppertight}.  After proving several technical inequalities, which include the use of the AM-GM inequality and Karamata's Inequality, in Subsection~\ref{inequal}, we complete the proof of our main result in Subsection~\ref{main} which we now state.

\begin{thm} \label{thm: complete}
For each $n, a \in \N$, $\chi_{\ell}(K_n \square K_{a,b}) = n+a$ if and only if $b \geq (n+a-1)!^a/(a-1)!^a$. That is, $f_a(K_n) =   \left( \frac{(n+a-1)!}{(a-1)!} \right)^a$ for each $n, a \in \N$.
\end{thm}

It is worth mentioning that when $n=1$,  Theorem~\ref{thm: complete} says $\chi_{\ell}(K_{a,b}) = \chi_{\ell}(K_1 \square K_{a,b}) = 1+a$ if and only if $b \geq a^a$ which is a well-known list coloring result. 

\section{Proof of Theorem~\ref{thm: complete}}

We now introduce some notation and terminology that will be used for the remainder of this paper. Suppose $a, b$, and $k \geq 2$ are positive integers and $M$ is a strongly $k$-chromatic-choosable graph. Suppose $H = M \square K_{a,b}$, $V(M) = \{v_1, \dots, v_n \}$, and the partite sets of the copy of $K_{a,b}$ used to form $H$ are $X = \{x_1, \dots, x_a \}$ and $Y = \{y_1, \dots, y_b \}$. Suppose $L$ is an arbitrary $(k+a-1)$-assignment for $H$.

Notice that for Theorem~\ref{thm: complete} we are specifically interested in the case in which $M = K_n$ for some $n \geq 2$. Note that $K_n$ is a strongly $n$-chromatic-choosable graph. If $n \geq a + 1$, $f_a(K_n) = \left( \frac{(n+a-1)!}{(a-1)!} \right)^a$ by Theorem~\ref{thm: sccexact}. Therefore, to prove Theorem~\ref{thm: complete}, we may suppose from this point forward that $n \leq a$ when $M = K_n$. Also, notice that by Theorem~\ref{thm: generalupper}, $f_a(K_n) \leq \left( \frac{(n+a-1)!}{(a-1)!} \right)^a$. So, proving Theorem~\ref{thm: complete} amounts to showing that if $b < \left( \frac{(n+a-1)!}{(a-1)!} \right)^a$, then $H$ is $(n+a-1)$-choosable.

\subsection{Important Tools} \label{reduce}

We begin by pointing out certain conditions on $L$ for which it is easy to construct a proper $L$-coloring of $H$.

\begin{lem} [\cite{KM19}] \label{lem: disjointlists}
Suppose $M$ is a strongly $k$-chromatic-choosable graph and $H= M \square K_{a,b}$.  Suppose that $L$ is a $(k+a-1)$-assignment for $H$ such that there exist $l, i,$ and $j$ with $i \neq j$ and $L(v_l, x_i) \cap L(v_l, x_j) \neq \emptyset$.  Then, there is a proper $L$-coloring of $H$.
\end{lem}

From Lemma~\ref{lem: disjointlists} we immediately get the following observation.

\begin{obs} \label{obs: disjointchoosable}
Suppose $M$ is a strongly $k$-chromatic-choosable graph and $H= M \square K_{a,b}$.  Suppose that $L$ is a $(k+a-1)$-assignment for $H$ such that the lists $L(v_i,x_1), \dots, L(v_i,x_a)$ are pairwise disjoint for each $i \in [n]$. If $H$ is $L$-colorable for any such $L$, then $H$ is $(k+a-1)$-choosable.
\end{obs}

So, Theorem~\ref{thm: complete} reduces to proving that there is a proper $L$-coloring of $H$ when $M = K_n$, $b < \left( \frac{(n+a-1)!}{(a-1)!} \right)^a$, and the lists $L(v_i, x_1), \dots, L(v_i, x_a)$ are pairwise disjoint for all $i \in [n]$. From now on we will assume that the lists $L(v_i, x_1), \dots, L(v_i, x_a)$ are pairwise disjoint for all $i \in [n]$.

Recall that for any $u \in V(M)$ (resp. $u \in X \cup Y$), we let $V_u$ be the subset of $V(H)$ consisting of the vertices with first (resp. second) coordinate $u$.  We also let $H_u = H[V_u]$.  Similarly, if $S \subseteq V(M) \cup (X \cup Y)$, we let $V_S = \bigcup_{s \in S} V_s$, $H_S = H[V_S]$, and $L_S$ be the list assignment for $H_S$ obtained by restricting the domain of $L$ to $V_S$.

Let $f$ be a proper $L_X$-coloring of $H_X$. We say $f$ is a \emph{bad coloring for $H_{y_i}$} if there is no proper $L'$-coloring for $H_{y_i}$ where $L'$ is the list assignment for $H_{y_i}$ given by $L'(v_j,y_i) = L_{y_i}(v_j,y_i) - \{f(v_j, x_l) : l \in [a] \}$ for each $j \in [n]$. Additionally, we say that $f$ is an \textit{$\mathit{(n-1)}$-to-1 coloring} if for each color $q$ in the range of $f$, $\card{f^{-1}(q)} \leq n - 1$.

Our next lemma relates the notion of bad coloring to the existence of a proper $L$-coloring of $H$.

\begin{lem} [\cite{KM19}] \label{lem: badcolor}
Suppose $H = M \square K_{a,b}$ with $a,b \in \N$ and $L$ is a list assignment for $H$.  Suppose $\mathcal{C}_X$ is the set of all proper $L_X$-colorings of $H_X$. For each $f \in \mathcal{C}_X$ there exists an $l \in [b]$ such that $f$ is a bad coloring for $H_{y_l}$ if and only if there is no proper $L$-coloring of $H$.
\end{lem}

Let $\mathcal{H} = \{H_{y_1},\dots, H_{y_b}\}$. Suppose $\mathcal{C}_X$ is the set of all proper $L_X$-colorings of $H_X$. The above lemma implies that when $H$ has no proper $L$-coloring, we can define a function $\mathfrak{F} \colon \mathcal{C}_X \to \mathcal{H}$ such that $\mathfrak{F}(c) = H_{y_l}$, where $H_{y_l}$ is the element of $\mathcal{H}$ with lowest index $l \in [b]$ such that $c$ is a bad coloring for $H_{y_l}$. Note that if we can show $\card{\mathfrak{F}^{-1}(H_{y_l})} \leq q$ for each $l \in [b]$, $\card{\mathcal{C}_X}/q \leq b$.

\begin{lem} [\cite{KM19}] \label{lem: scclower}
Suppose $M$ is strongly $k$-chromatic-choosable and $H= M \square K_{a, 1}$.  Let $L$ be a $(k+a-1)$-assignment for $H$ such that the lists $L(v_i,x_1), \dots, L(v_i,x_a)$ are pairwise disjoint for each $i \in [n]$.  Let $\mathcal{B}$ be the set of proper $L_X$-colorings of $H_X$ that are bad for $H_{y_1}$.  Then, $|\mathcal{B}| \leq 2^{k-1}$.
\end{lem}

Assuming the same setup as Lemma~\ref{lem: scclower}, we now prove that if there is an $(n-1)$-to-1 proper $L_X$-coloring of $H_X$ that is bad for $H_{y_1}$, then it is in fact the only bad coloring for $H_{y_1}$.

\begin{lem} \label{lem: geninjcolor}
Suppose $M$ is strongly $k$-chromatic-choosable, and $H= M \square K_{a, 1}$.  Let $L$ be a $(k+a-1)$-assignment for $H$ such that the lists $L(v_i,x_1), \dots, L(v_i,x_a)$ are pairwise disjoint for each $i \in [n]$.  Let $\mathcal{B}$ be the set of proper $L_X$-colorings of $H_X$ that are bad for $H_{y_1}$, and let $\mathcal{B}_I$ consist of all the elements of $\mathcal{B}$ that are $(n-1)$-to-1. If $\mathcal{B}_I$ is non-empty, then $\card{\mathcal{B}_I} = 1$ and $\mathcal{B} = \mathcal{B}_I$.
\end{lem}
\begin{proof}
    Suppose for the sake of contradiction that there exist two different colorings $c, c' \in \mathcal{B}$ such that $c \in \mathcal{B}_I$. Since $M$ is strongly $k$-chromatic-choosable and both $c$ and $c'$ are bad for $H_{y_1}$, there exist two sets of colors $K, K'$, both of size $k-1$, such that for all $i \in [n]$,
    \begin{equation} \label{eq: listdecomp}
        L(v_i, y_1) - \{c(v_i, x_1), \dots, c(v_i, x_a)\} = K,
    \end{equation}
    and
    \begin{equation} \label{eq: listdecompprime}
        L(v_i, y_1) - \{c'(v_i, x_1), \dots, c'(v_i, x_a)\} = K'.
    \end{equation}
    Note that $K, K' \subseteq L(v_i, y_1)$ for each $i \in [n]$. We will obtain a contradiction when $K = K'$ by showing that $c = c'$.  We will obtain a contradiction when $K \ne K'$ by showing that $c \not\in \mathcal{B}_I$.
    
    If $K = K'$, then for all $i \in [n]$, 
    \[\{c(v_i, x_1), \dots, c(v_i, x_a)\} = \{c'(v_i, x_1), \dots, c'(v_i, x_a)\}. \]
    
   \noindent Fix an arbitrary $j \in [a]$ and $s \in [n]$. Since the lists $L(v_s,x_1), \dots, L(v_s,x_a)$ are pairwise disjoint, for any $l \in [a]$ such that $l \ne j$, we have $c(v_s, x_j) \ne c'(v_s, x_l)$. Therefore $c(v_s, x_j) = c'(v_s, x_j)$. As $j$ and $s$ were arbitrary, we conclude that $c = c'$.

    Now assume $K \ne K'$. Since $\card{K} = \card{K'}$, there exists a color $q \in K' - K$. Note that for all $i \in [n]$, $q \in L(v_i, y_1)$ since $K' \subseteq L(v_i, y_1)$. Then, for all $i \in [n]$, $q \in \{c(v_i, x_1), \dots, c(v_i, x_a)\}$ by (\ref{eq: listdecomp}) since $q \not\in K$. This means for each $i \in [n]$, there exists a $j \in [a]$ that satisfies $c(v_i, x_j) = q$. Therefore, $|c^{-1}(q)| \geq n$ contradicting $c \in \mathcal{B}_I$.

    Thus, it is impossible to have two different colorings in $\mathcal{B}$ when one of them is in $\mathcal{B}_I$. 
\end{proof}

Now suppose $H = M \square K_{a, b}$ and $L$ is a $(k+a-1)$-assignment for $H$ that satisfies the conditions we have established. In the next lemma, we give a sufficient condition in terms of a bound on $b$ for there to be a proper $L$-coloring of $H$. The bound on $b$ is in terms of the number of proper $L_X$-colorings of $H_X$ and the number of $(n-1)$-to-1 proper $L_X$-colorings of $H_X$.

\begin{lem} \label{lem: gencolorbound}
    Suppose $M$ is strongly $k$-chromatic-choosable, and $H= M \square K_{a, b}$.  Let $L$ be a $(k+a-1)$-assignment for $H$ such that the lists $L(v_i,x_1), \dots, L(v_i,x_a)$ are pairwise disjoint for each $i \in [n]$. Let $\mathcal{C}_X$ be the set of all proper $L_X$-colorings of $H_X$, and let $\mathcal{I}_X \subseteq \mathcal{C}_X$ be the set of $(n-1)$-to-1 colorings in $\mathcal{C}_X$. If
    \[b < \card{\mathcal{I}_X} + \frac{\card{\mathcal{C}_X} - \card{\mathcal{I}_X}}{2^{k-1}}\]
    then $H$ has a proper $L$-coloring.
\end{lem}
\begin{proof}
    We prove the contrapositive. Let $\mathcal{H} = \{H_{y_1},\dots, H_{y_b}\}$. Let $\mathfrak{F} \colon \mathcal{C}_X \to \mathcal{H}$ be the function given by $\mathfrak{F}(c) = H_{y_l}$ where $H_{y_l}$ is the element of $\mathcal{H}$ with lowest index $l \in [b]$ such that $c$ is a bad coloring for $H_{y_l}$. Such a $y_l$ always exists due to Lemma~\ref{lem: badcolor}. We also let $\overline{\mathcal{I}_X} = \mathcal{C}_X - \mathcal{I}_X$.

    By Lemma~\ref{lem: geninjcolor}, when the domain of $\mathfrak{F}$ is restricted to $\mathcal{I}_X$ the resulting function is injective. Thus, $\card{\mathfrak{F}(\mathcal{I}_X)} = \card{\mathcal{I}_X}$. On the other hand, Lemma~\ref{lem: scclower} implies that for each element in $\mathfrak{F}(\overline{\mathcal{I}_X})$, there are at most $2^{k-1}$ distinct elements from $\overline{\mathcal{I}_X}$ that are mapped to it. Hence $2^{k-1}\card{\mathfrak{F}(\overline{\mathcal{I}_X})} \geq \card{\overline{\mathcal{I}_X}}$ implying $\card{\mathfrak{F}(\overline{\mathcal{I}_X})} \geq \card{\overline{\mathcal{I}_X}}/2^{k-1}$. By Lemma~\ref{lem: geninjcolor}, the images $\mathfrak{F}(\mathcal{I}_X)$ and $\mathfrak{F}(\overline{\mathcal{I}_X})$ are disjoint. Therefore, $\card{\mathfrak{F}(\mathcal{C}_X)} = \card{\mathfrak{F}(\mathcal{I}_X)} + \card{\mathfrak{F}(\overline{\mathcal{I}_X})}$ which implies
        $$\card{\mathfrak{F}(\mathcal{C}_X)} \geq \card{\mathcal{I}_X} + \frac{\card{\overline{\mathcal{I}_X}}}{2^{k-1}} = \card{\mathcal{I}_X} + \frac{\card{\mathcal{C}_X} - \card{\mathcal{I}_X}}{2^{k-1}}.$$
    On the other hand, since $\mathfrak{F}(\mathcal{C}_X) \subseteq \mathcal{H}$, $\card{\mathfrak{F}(\mathcal{C}_X)} \leq \card{\mathcal{H}} = b$; hence, the result follows.
\end{proof}

We can now describe our strategy for proving Theorem~\ref{thm: complete}.  We suppose $M = K_n$, $n \leq a$, and $H = M \square K_{a,b}$ where $b = P_{\ell}(M, n+a-1)^a - 1$. Then, we suppose that there is an $(n+a-1)$-assignment $L$ for $H$ for which there is no proper $L$-coloring of $H$. Observation~\ref{obs: disjointchoosable} allows us to assume that the lists $L(v_i, x_1), \dots, L(v_i, x_a)$ are pairwise disjoint for each $i \in [n]$.  We show that under these conditions $P_{\ell}(M, n+a-1)^a  \leq \card{\mathcal{I}_X} + (\card{\mathcal{C}_X} - \card{\mathcal{I}_X})/2^{n-1}$. Then, Lemma~\ref{lem: gencolorbound} implies $f_a(M) \geq P_{\ell}(M, n+a-1)^a$. This along with Theorem~\ref{thm: generalupper} implies Theorem~\ref{thm: complete}.

\subsection{Technical Lemmas} \label{inequal}

For the remainder of the paper, assume $M = K_n$ and $2 \leq n \leq a$. Note that $M$ is strongly $n$-chromatic-choosable. Also, assume $H = M \square K_{a,b}$, where $b = P_{\ell}(K_n, n+a-1)^a - 1 = (\prod_{i=0}^{n-1} (n+a-1-i))^a - 1$. Assume $L$ is an $(n+a-1)$-assignment for $H$ such that the lists $L(v_i,x_1), \dots, L(v_i,x_a)$ are pairwise disjoint for each $i \in [n]$.  Finally, assume $\mathcal{C}_X$ is the set of all proper $L_X$-colorings of $H_X$, and assume $\mathcal{I}_X$ is the set of $(n-1)$-to-1 colorings in $\mathcal{C}_X$.

\begin{lem} \label{lem: atmostn}
Suppose $c \in \mathcal{C}_X$ and $s$ is any color in the range of $c$. If there are distinct vertices $(v_i, x_j), (v_{i'}, x_{j'}) \in c^{-1}(s)$, then $i \ne i'$ and $j \ne j'$. Consequently, $\card{c^{-1}(s)} \leq n$.
\end{lem}
\begin{proof}
    Suppose there are distinct vertices $(v_i, x_j), (v_{i'}, x_{j'}) \in c^{-1}(s)$.  Note that if $j = j'$, then $i \neq i'$ and $c(v_i,x_j) = c(v_{i'},x_{j})$.  This contradicts the fact that $c$ is proper.  Thus, $j \neq j'$.  Finally, if $i = i'$, then $s \in L(v_i,x_j) \cap L(v_i,x_{j'})$ which contradicts the fact that $L(v_i,x_{j})$ and $L(v_i,x_{j'})$ are disjoint. 
\end{proof}

We now establish some notation. For each $q \in \bigcup_{j=1}^n L(v_1, x_j)$, we let

\[s_q = \begin{cases*}
    1, &if there exists $c \in \mathcal{C}_X$ such that $\card{c^{-1}(q)} = n$; \\
    0, &otherwise.
\end{cases*}\]
Additionally, for each $\mathbf{q} = (q_1, \dots, q_a) \in \prod_{j=1}^a L(v_1, x_j)$, we let $s(\mathbf{q}) = \sum_{j=1}^a s_{q_i}$. We also define $\mathcal{C}_{X, \mathbf{q}}$ as the set of all proper colorings $c$ of $H_X$ such that for all $j \in [a]$, $c(v_1, x_j) = q_j$, and $\mathcal{I}_{X, \mathbf{q}}$ as the set of $(n-1)$-to-1 proper colorings $c$ of $H_X$ such that for all $j \in [a]$, $c(v_1, x_j) = q_j$.

We can readily observe that 

\[\sum_{\mathbf{q} \in \prod_{j=1}^a L(v_1, x_j)} \card{\mathcal{I}_{X, \mathbf{q}}} = \card{\mathcal{I}_X} \qquad \text{and} \qquad \sum_{\mathbf{q} \in \prod_{j=1}^a L(v_1, x_j)} \card{\mathcal{C}_{X, \mathbf{q}}} = \card{\mathcal{C}_X}.\]
We will use these identities to bound $\card{\mathcal{C}_X}$ and $\card{\mathcal{I}_X}$. First, we need a technical lemma.

\begin{lem} \label{lem: ntimes}
    Let $\mathbf{q} = (q_1, \dots, q_a)$ be a fixed element of  $\prod_{j=1}^a L(v_1, x_j)$. If $s_{q_t} = 1$ for some $t \in [a]$, then $q_t \notin L(v_i,x_t)$ for each $i \in \{2, \dots, n\}$.
\end{lem}
\begin{proof}
    For the sake of contradiction, suppose that $q_t \in L(v_r, x_t)$ for some $r \in \{2, \dots, n\}$. Since $s_{q_t} = 1$, there exists a coloring $c \in \mathcal{C}_X$ such that $\card{c^{-1}(q_t)} = n$ and $c(v_1, x_t) = q_t$. Lemma~\ref{lem: atmostn} implies $\{ i \in [n] : \text{there is a $j \in [a]$ such that $c(v_i,x_j) = q_t$}\} = [n]$.  So, there exists a $t' \in [a]$ such that $c(v_r, x_{t'}) = q_t$, and the properness of $c$ implies that $t \ne t'$. Therefore, $L(v_r, x_{t'}) \cap L(v_r, x_t) \ne \emptyset$, a contradiction.
\end{proof}

\begin{lem} \label{lem: coloringbound}
    Let $\mathbf{q} = (q_1, \dots, q_a)$ be a fixed element of  $\prod_{j=1}^a L(v_1, x_j)$, and let $s = s(\mathbf{q})$. Then 
    \[\card{\mathcal{C}_{X, \mathbf{q}}} \geq a^{a-s} (n+a-1)^s \prod_{i=2}^{n-1} (n+a-i)^a. \]
\end{lem}

\begin{proof}
We establish the desired bound by giving a lower bound on the number of ways to greedily complete a proper $L_X$-coloring of $H_X$, $c$, where for each $j \in [a]$, $c(v_1, x_j) = q_j$ (i.e., the vertices $(v_1,x_1), \dots, (v_1,x_a)$ are precolored according to \textbf{q}).

Suppose $j \in [a]$.  Now, consider the number of ways we can greedily construct a proper $L_{x_j}$-coloring of $H_{x_j}$ that colors $(v_1,x_j)$ with $q_j$.  If $s_{q_j} = 1$, then $q_j \notin L(v_i,x_j)$ for each $i \in \{2, \dots, n\}$ by Lemma~\ref{lem: ntimes}. Consequently, there are at least $\prod_{i=2}^n (n+a - i +1)$ ways to greedily complete a proper $L_{x_j}$-coloring of $H_{x_j}$.  On the other hand, if $s_{q_j} = 0$, we can only guarantee that there are at least $\prod_{i=2}^n (n+a - i)$ ways to greedily complete a proper $L_{x_j}$-coloring of $H_{x_j}$. It follows that

\[\card{\mathcal{C}_{X, \mathbf{q}}} \geq \left (\prod_{i=2}^n (n+a - i +1) \right )^s \left (\prod_{i=2}^n (n+a - i) \right )^{a-s} = a^{a-s} (n+a-1)^s \prod_{i=2}^{n-1} (n+a-i)^a.\]
\end{proof}

Next, we define an auxiliary graph that we will use in the next two lemmas to get to a lower bound on $\card{\mathcal{I}_{X, \mathbf{q}}}$.

\begin{defin}
   For each $\mathbf{q} = (q_1, \dots, q_a)$ in  $\prod_{j=1}^a L(v_1, x_j)$, we define a graph $M_{ \mathbf{q}}$. Let $V(M_{\mathbf{q}}) = V_X$. The edge set of $M_{\mathbf{q}}$ is such that the following conditions hold. For each $j \in [a]$, if $s_{q_j} = 0$, the set $\{(v_i, x_j)\colon i \in [n]\}$ is a clique in $M_{ \mathbf{q}}$. Otherwise, $(v_1, x_j)$ is adjacent to each vertex in the set $\{(v_2, x_{j'}) \colon j' \in [a], j' \ne j\}$, and $\{(v_i, x_j)\colon 2 \leq i \leq n\}$ is a clique in $M_{\mathbf{q}}$. 
\end{defin}

\begin{lem} \label{lem: coloringcorrespondence}
    Suppose $\mathbf{q} = (q_1, \dots, q_a) \in \prod_{j=1}^a L(v_1, x_j)$. Let $\mathcal{C}_{ \mathbf{q}}'$ be the set of proper $L_X$-colorings $c$ of $M_{\mathbf{q}}$ such that for all $j \in [a]$, $c(v_1, x_j) = q_j$. Then $\mathcal{C}_{ \mathbf{q}}' \subseteq \mathcal{I}_{X, \mathbf{q}}$.
\end{lem}
\begin{proof}
 Suppose $f$ is an arbitrary element of $\mathcal{C}_{ \mathbf{q}}'$. We claim that $f$ is an $(n-1)$-to-1 proper $L_X$-coloring of $H_X$. 

 We begin by showing that $f$ is a proper $L_X$-coloring of $H_X$. Suppose $i, i' \in [n]$ with $i < i'$ and $j \in [a]$.  If $s_{q_j} = 0$, we immediately have that $f(v_i,x_j) \neq f(v_{i'},x_j)$ since $(v_i,x_j)(v_{i'},x_j) \in E(M_{\textbf{q}})$.  So, we may suppose that $s_{q_j} = 1$. If $i \geq 2$, we once again have $f(v_i,x_j) \neq f(v_{i'},x_j)$ since $(v_i,x_j)(v_{i'},x_j) \in E(M_{\textbf{q}})$.  So, suppose that $i=1$.  Lemma~\ref{lem: ntimes} tells us that $q_j$ is not in any of the lists: $L(v_2,x_j), \dots, L(v_n,x_j)$ which implies $f(v_1,x_j) \neq f(v_{i'},x_j)$.  Thus, $f$ is a proper $L_X$-coloring of $H_X$.

Finally we must show that $f$ uses no color more than $(n-1)$ times.  For the sake of contradiction suppose there is a $\gamma$ such that $|f^{-1}(\gamma)| = n$. By Lemma~\ref{lem: atmostn}, $\{ i \in [n] : \text{there is a $j \in [a]$ such that $f(v_i,x_j) = \gamma$}\} = [n]$.

Without loss of generality suppose $f(v_1,x_1)=\gamma$.  This means that $q_1=\gamma$.  Since $f$ is a proper $L_X$-coloring of $H_X$, we know $s_{\gamma} = 1$.  We also know there is an $\omega \in [a]$ such that $f(v_2,x_{\omega}) = \gamma$.  By the manner in which the edges of $M_{\textbf{q}}$ are defined it must be that $\omega = 1$.  However, Lemma~\ref{lem: ntimes} tells us $\gamma \notin L(v_2,x_1)$ which is a contradiction.
\end{proof}

\begin{lem} \label{lem: muchhardercoloringbound}
    Let $\mathbf{q} = (q_1, \dots, q_a)$ be a fixed element of  $\prod_{j=1}^a L(v_1, x_j)$, and let $s = s(\mathbf{q})$. Let $\mathcal{C}_{ \mathbf{q}}'$ be the set of proper $L_X$-colorings $c$ of $M_{\mathbf{q}}$ such that for all $j \in [a]$, $c(v_1, x_j) = q_j$.  For each $j \in [a]$ let $d_j = |L(v_2,x_j) \cap (\{q_i : s_{q_i}=1\} \cup \{q_j\})|$.  The following statements hold.
    \begin{enumerate}[i)]

    \item We have $d_j \leq s+1$ for each $j \in [a]$ and $\sum_{j=1}^a d_j \leq a.$
    \item It is the case that
    \[\card{\mathcal{C'_\mathbf{q}}} \geq \prod_{j=1}^a (n+a-1 - d_j) \left (\prod_{i=3}^n (n+a - i +1) \right )^s \left (\prod_{i=3}^n (n+a - i) \right )^{a-s}.\]
    \end{enumerate}
\end{lem}

\begin{proof}
  The first inequality of Statement~(i) follows from the definition of $d_j$.  The second inequality follows from the fact that $L(v_2,x_1), \ldots, L(v_2,x_a)$ are pairwise disjoint and $|\{q_1, \ldots, q_a\}| = a$.  

  Now, we turn our attention to Statement~(ii).  We prove our desired bound by describing a procedure for greedily constructing a proper $L_X$-coloring, $c$, of $M_{\textbf{q}}$ and bounding the number of ways each step can completed.

First, for each $j \in [a]$, let $c(v_1, x_j) = q_j$
(i.e., the vertices $(v_1,x_1), \dots, (v_1,x_a)$ are colored according to \textbf{q}).  This can be done in one way.  Then, for each $j \in [a]$ color $(v_2,x_j)$ with some $a_j \in L(v_2,x_j) - (\{q_i : s_{q_i}=1\} \cup \{q_j\})$.  This can be done in at least $\prod_{j=1}^a (n+a-1 - d_j)$ ways. Notice our coloring is now complete if $n =2$ and the desired bound holds when $n =2$.  So, we may assume $n \geq 3$.

Suppose $j \in [a]$.  Now, consider the number of ways we can greedily construct a proper $L_{x_j}$-coloring of $H_{x_j}$ that colors $(v_1,x_j)$ with $q_j$ and $(v_2, x_j)$ with $a_j$.  If $s_{q_j} = 1$, then $q_j \notin L(v_i,x_j)$ for each $i \in \{3, \dots, n\}$ by Lemma~\ref{lem: ntimes}. Consequently, there are at least $\prod_{i=3}^{n} (n+a - 1 - (i-2))$ ways to greedily complete a proper $L_{x_j}$-coloring of $H_{x_j}$.  On the other hand, if $s_{q_j} = 0$, we can only guarantee that there are at least $\prod_{i=3}^n (n+a - 1 - (i-1))$ ways to greedily complete a proper $L_{x_j}$-coloring of $H_{x_j}$. It follows that
\begin{align*}
    \card{\mathcal{C'_\mathbf{q}}} \geq \prod_{j=1}^a (n+a-1 - d_j) \left (\prod_{i=3}^n (n+a - i +1) \right )^s \left (\prod_{i=3}^n (n+a - i) \right )^{a-s}.
\end{align*} \end{proof}

\subsubsection{Applying Karamata's Inequality}\label{sec: Karamata}

Lemmas~\ref{lem: coloringcorrespondence} and \ref{lem: muchhardercoloringbound} yield a lower bound for $\card{\mathcal{I}_{X, \mathbf{q}}}$ that depends on $d_1, \ldots, d_a$. We will now use the celebrated Karamata's Inequality~\cite{KD05} to obtain a lower bound that depends only on $n, a$, and $s$. The statement of this inequality first requires a definition. Let $\textbf{a} = (a_i)_{i=1}^n$ and $\textbf{b} = (b_i)_{i=1}^n$ be two finite sequences of real numbers with $n \geq 2$. We say that $\textbf{a}$ \textit{majorizes} $\textbf{b}$, written $\textbf{a} \succ \textbf{b}$, if the following three conditions hold:
\begin{enumerate}[i)]
    \item $a_1 \geq \dots \geq a_n$ and $b_1 \geq \dots \geq b_n$;
    \item $a_1 + \dots + a_k \geq b_1 + \dots + b_k$ for each $k \in [n-1]$;
    \item $a_1 + \dots + a_n = b_1 + \dots + b_n$.
\end{enumerate}

\begin{lem} [\cite{KD05}]
    Suppose $n \geq 2$. Let $\mathbf{a} = (a_i)_{i=1}^n$ and $\mathbf{b} = (b_i)_{i=1}^n$ be finite sequences of real numbers from an interval $(\alpha, \beta) \subseteq \R$. If $\mathbf{a} \succ \mathbf{b}$ and $f \colon (\alpha, \beta) \to \R$ is a concave function, then
    \[\sum_{i=1}^n f(a_i) \leq \sum_{i=1}^n f(b_i).\]
\end{lem}

\begin{lem} \label{lem: optlemma}
    Let $n, m, k, C$ be positive integers such that $n \geq 2$, $m > k$ and $C > k$. Let $\mathcal{X}$ be the set of $n$-tuples $(x_1, \dots, x_n) \in \Z^n$ such that $0 \leq x_i \leq k$ for all $i \in [n]$ and $x_1 + \dots + x_n \leq m$. Suppose $m = kq + r$ where $q$ and $r$ are nonnegative integers such that $0 \leq r < k$. Then the following statements hold.
    \begin{enumerate}[i)]
        \item If $n \leq q$, then
        \[\min_{(x_1, \dots, x_n) \in \mathcal{X}} \prod_{i=1}^n (C - x_i) = (C-k)^n.\]
        \item If $n \geq q+1$, then
        \[\min_{(x_1, \dots, x_n) \in \mathcal{X}} \prod_{i=1}^n (C - x_i) = (C-k)^q(C-r)C^{n-(q+1)}.\]
    \end{enumerate}
\end{lem}
\begin{proof}
    For all $(x_1, \dots, x_n) \in \mathcal{X}$, let

    \[P(x_1, \dots, x_n) = \prod_{i=1}^n (C - x_i).\]
    
    First, suppose $n \leq q$. Let $\textbf{y}=(y_1, \dots, y_n)$ be the $n$-tuple given by $y_i=k$ for each $i \in [n]$. Since $kn \leq kq \leq m$, $\mathbf{y} \in \mathcal{X}$.  For any $(x_1, \dots, x_n) \in \mathcal{X}$, $C - x_i \geq C - k$ for each $i \in [n]$. So,
    \[P(x_1, \dots, x_n) \geq P(\mathbf{y}) = (C - k)^n\]
    which completes the proof of Statement~(i).

    From now on, suppose $n \geq q + 1$. We claim that the minimum value of $P(x_1, \dots, x_n)$ over $\mathcal{X}$ is attained at a point $(z_1, \dots, z_n)$ in which $z_1 + \dots + z_n = m$. Note that if $z_1 + \dots + z_n < m$, there is a $t \in [n]$ such that $z_t < k$; otherwise,

    \[m > z_1 + \dots + z_n = kn \geq kq + k > kq + r = m\]
    which is a contradiction. Let $(w_1, \dots, w_n)$ be the $n$-tuple given by $w_i = z_i$ when $i \neq t$ and $w_t = z_t+1$. 
 Then $(w_1, \dots, w_n) \in \mathcal{X}$, since $w_t = z_t + 1 \leq k$ and $w_1 + \dots + w_n \leq m$. Furthermore,
    \[P(w_1, \dots, w_n) < P(z_1, \dots, z_n).\]

    Therefore, if $z_1 + \dots + z_n < m$, the minimum value of $P$ is not attained at $(z_1, \dots, z_n)$. Hence the minimum occurs at an element in the set $\mathcal{X'}$, where $\mathcal{X'}$ is the set of all elements of $\mathcal{X}$ whose coordinates sum to $m$.

    Let $f(x) = \log (C - x)$ for all $x \in (-1, C)$. Suppose $(x_1, \dots, x_n) \in \mathcal{X}$.  Since $x_i \leq k < C$ for all $i \in [n]$, $f(x_i)$ is real for each $i \in [n]$. Also, for each $(x_1, \dots, x_n) \in \mathcal{X}$, $\log P(x_1, \dots, x_n) = \sum_{i=1}^n f(x_i).$

    Consider the $n$-tuple $(x_1^*, \ldots, x_n^*)$ where $x_i^* = k$ for each $i \in [q]$, $x_{q+1}^* = r$, and $x_i^* = 0$ otherwise. Additionally, for a given $(x_1, \dots, x_n) \in \mathcal{X'}$, let $x_1', \dots, x_n'$ be an ordering of the numbers $x_1, \dots, x_n$ such that $x_1' \geq \dots \geq x_n'$. We claim $(x_1^*, \dots, x_n^*) \succ (x_1', \dots, x_n')$. To see why, note: for all $l \in [q]$, $\sum_{i=1}^l x_i^* = kl \geq \sum_{i=1}^l x_i'$, for all $l \in \{q+1, \ldots, n-1\}$ $\sum_{i=1}^l x_i^* = m \geq \sum_{i=1}^l x_i'$, and $\sum_{i=1}^n x_i^* = m = \sum_{i=1}^n x_i'.$

    Since $f(x)$ is concave on $(-1, C)$, Karamata's inequality yields
    \[\sum_{i=1}^n f(x_i) = \sum_{i=1}^n f(x_i') \geq \sum_{i=1}^n f(x_i^*) = q \log (C - k) + \log (C - r) + (n-q-1) \log C \]
    which implies
    \[P(x_1, \dots, x_n) \geq P(x_1^*, \dots, x_n^*) = (C-k)^q(C-r)C^{n-q-1}\]
    as desired.
\end{proof}

\begin{lem} \label{lem: nminus1to1bound}
    Let $\mathbf{q} = (q_1, \dots, q_a)$ be a fixed element of $\prod_{j=1}^a L(v_1, x_j)$, and let $s = s(\mathbf{q})$. Then 
    \[\card{\mathcal{I}_{X, \mathbf{q}}} \geq (n+a-s-2)^{\frac{a}{s+1}} (n+a-1)^{\frac{sa}{s+1}} \left (\prod_{i=3}^n (n+a - i +1) \right )^s \left (\prod_{i=3}^n (n+a - i) \right )^{a-s}.\]
\end{lem}
\begin{proof}
    Using the notation of Lemma~\ref{lem: muchhardercoloringbound}, by Statement~(ii) of Lemma~\ref{lem: muchhardercoloringbound}  it suffices to show that 
    \[\prod_{j=1}^a (n+a-1 - d_j) \geq (n+a-s-2)^{\frac{a}{s+1}} (n+a-1)^{\frac{sa}{s+1}}.\]

    Suppose $a = (s+1)q + r$, where $q$ and $r$ are nonnegative integers such that $0 \leq r < s+1$. Suppose $s = 0$. Then, Statement (i) of Lemma~\ref{lem: muchhardercoloringbound} and Lemma~\ref{lem: optlemma} imply

    \begin{align*}
        \prod_{j=1}^a (n+a-1 - d_j) &\geq (n+a-2)^a
    \end{align*}
    as desired. Now suppose $0 < s \leq a$. Statement (i) of Lemma~\ref{lem: muchhardercoloringbound} and Lemma~\ref{lem: optlemma} imply

    \begin{align*}
        \prod_{j=1}^a (n+a-1 - d_j) &\geq (n+a-2-s)^q (n+a-1-r)(n+a-1)^{a-(q+1)}.
    \end{align*}
    By the AM-GM inequality,
    \[n+a-1-r = \frac{r(n+a-2-s) + (s+1-r)(n+a-1)}{s+1} \geq (n+a-2-s)^{\frac{r}{s+1}} (n+a-1)^{\frac{s+1-r}{s+1}}.\]
    Therefore,
    \begin{align*}
        \prod_{j=1}^a (n+a-1 - d_j) &\geq (n+a-2-s)^{q+\frac{r}{s+1}}(n+a-1)^{a-q-\frac{r}{s+1}} \\
        &= (n+a-s-2)^{\frac{a}{s+1}} (n+a-1)^{\frac{sa}{s+1}}
    \end{align*}
    as desired.
\end{proof}

\subsection{Completing the proof of Theorem~\ref{thm: complete}} \label{main}

In this section we will prove that $P_{\ell}(M, n+a-1)^a  \leq \card{\mathcal{I}_X} + (\card{\mathcal{C}_X} - \card{\mathcal{I}_X})/2^{n-1}$ which by Lemma~\ref{lem: gencolorbound} will complete our proof of Theorem~\ref{thm: complete}.  Note that $P_{\ell}(M, n+a-1)^a = \prod_{i=1}^n (n+a-i)^a$, and
$$ \frac{\card{\mathcal{C}_X} + (2^{n-1} - 1)\card{\mathcal{I}_X}}{2^{n-1}} 
        = \frac{1}{2^{n-1}} \sum_{\mathbf{q} \in \prod_{j=1}^a L(v_1, x_j)} (\card{\mathcal{C}_{X, \mathbf{q}}} + (2^{n-1} - 1)\card{\mathcal{I}_{X, \mathbf{q}}}).$$
Since the sum on the right has $(n+a-1)^a$ terms, showing that $\card{\mathcal{C}_{X, \mathbf{q}}} + (2^{n-1} - 1)\card{\mathcal{I}_{X, \mathbf{q}}} \geq 2^{n-1} \prod_{i=2}^n (n+a-i)^a$ for each $\mathbf{q} \in \prod_{j=1}^a L(v_1, x_j)$ will imply the desired inequality.  If $s = s(\mathbf{q})$ for some $\mathbf{q} \in \prod_{j=1}^a L(v_1, x_j)$, Lemmas~\ref{lem: coloringbound} and \ref{lem: nminus1to1bound} along with some simplification tell us that proving
 \begin{align} &\left (1 + \frac{n-1}{a} \right )^{s} + (2^{n-1} -1) \left [ \left (1 - \frac{s}{n+a-2} \right ) \left (1+\frac{1}{n+a-2} \right )^s \right ]^{\frac{a}{s+1}} \left (1 + \frac{n-2}{a} \right )^{s} \label{key} \\ 
 &\geq 2^{n-1} \nonumber
 \end{align}
will imply $\card{\mathcal{C}_{X, \mathbf{q}}} + (2^{n-1} - 1)\card{\mathcal{I}_{X, \mathbf{q}}} \geq 2^{n-1} \prod_{i=2}^n (n+a-i)^a$.  So, our goal is to prove~(\ref{key}) whenever $2 \leq n \leq a$ and $0 \leq s \leq a$.

We will first prove (\ref{key}) when $n \geq 3$.  Our first lemma shows that when $s$ is large the first term on the left side (\ref{key}) is large enough to justify (\ref{key}).

\begin{lem} \label{lem: bigsineq}
    Suppose $a \geq n \geq 3$ and $s > 0.73(n+a-2)$. Then,
    \[\left (1 + \frac{n-1}{a} \right )^{s} > 2^{n-1}.\]
\end{lem}
\begin{proof}
   We will use the following well-known inequality~(\cite{M70}, page 267): for any real numbers $x, n > 0$ it is the case that
    \[\left (1 + \frac{x}{n} \right )^{n+\frac{x}{2}} \geq e^x.\]
    
    Since $n+a-2 \geq a+(n-1)/2$, the above inequality implies
    \[\left (1 + \frac{n-1}{a} \right )^{n+a-2} \geq \left (1 + \frac{n-1}{a} \right )^{a+(n-1)/2} \geq e^{n-1}.\]
    Therefore,
    \begin{align*}
        \left (1 + \frac{n-1}{a} \right )^{s} > \left (1 + \frac{n-1}{a} \right )^{0.73(n+a-2)} \geq e^{0.73(n-1)} > 2^{n-1}.
    \end{align*}
\end{proof}

Suppose $x \in \mathbb{R}$.  One can easily show with basic calculus-based arguments that $\ln(1+x) \geq x - x^2/2$ whenever $x \geq 0$ and $\ln(1+x) \geq x - 1.1x^2$ whenever $x \geq -0.73$.  Also, by the AM-GM Inequality, 
\begin{align*}
        &\left (1 + \frac{n-1}{a} \right )^{s} + (2^{n-1} -1) \left [ \left (1 - \frac{s}{n+a-2} \right ) \left (1+\frac{1}{n+a-2} \right )^s \right ]^{\frac{a}{s+1}} \left (1 + \frac{n-2}{a} \right )^{s} \\
        &\geq 2^{n-1} \\
        &\left [\left (1 + \frac{n-1}{a} \right )^{s(s+1)} \left (1 + \frac{n-2}{a} \right )^{(2^{n-1} -1)s(s+1)}  \left [ \left (1 - \frac{s}{n+a-2} \right ) \left (1+\frac{1}{n+a-2} \right )^s \right ]^{(2^{n-1} -1) a} \right ]^{\frac{1}{(s+1)2^{n-1}}}.
    \end{align*}
We will need these facts to handle small values of $s$ when $n \geq 3$.

\begin{lem} \label{lem: smallsineq}
    Suppose $n \geq 3$ and $0 \leq s \leq 0.73(n+a-2)$. Then
    \[\left (1 + \frac{n-1}{a} \right )^{s(s+1)} \left (1 + \frac{n-2}{a} \right )^{(2^{n-1} -1)s(s+1)} \left [ \left (1 - \frac{s}{n+a-2} \right ) \left (1+\frac{1}{n+a-2} \right )^s \right ]^{(2^{n-1} -1) a} \geq 1.\]
\end{lem}
Notice that Lemma~\ref{lem: bigsineq} along with our application of the AM-GM inequality and
Lemma~\ref{lem: smallsineq} will imply (\ref{key}) when $n \geq 3$. \begin{proof}
     Clearly, the desired inequality is equivalent to
    \begin{align*}
        &s(s+1)\ln \left (1+ \frac{n-1}{a} \right) + (2^{n-1} - 1)s(s+1)\ln \left (1 + \frac{n-2}{a} \right ) \\ 
    &+ (2^{n-1} - 1)a\ln \left (1 - \frac{s}{n+a-2} \right) + (2^{n-1} - 1)as\ln \left (1 + \frac{1}{n+a-2} \right) \geq 0.
    \end{align*}
We will now prove this inequality for $n \geq 4$, and then we will handle $n=3$ separately.  Let $M = 2^{n-1}-1$ and $b = s(s+1)$.  For $n \geq 4$, we see
  \begin{align*}
        &s(s+1)\ln \left (1+ \frac{n-1}{a} \right) + (2^{n-1} - 1)s(s+1)\ln \left (1 + \frac{n-2}{a} \right ) \\ 
    &+ (2^{n-1} - 1)a\ln \left (1 - \frac{s}{n+a-2} \right) + (2^{n-1} - 1)as\ln \left (1 + \frac{1}{n+a-2} \right) \\
    &\geq Mb\ln \left (1 + \frac{n-2}{a}  \right ) +Ma \left( \ln \left (1 - \frac{s}{n+a-2} \right) + s\ln \left (1 + \frac{1}{n+a-2} \right) \right) \\
    &\geq Mb\left (\frac{n-2}{a}-\frac{(n-2)^2}{2a^2}  \right ) +Ma \left( \left (\frac{-s}{n+a-2}-\frac{1.1s^2}{(n+a-2)^2} \right) +  \left (\frac{s}{n+a-2} - \frac{s}{2(n+a-2)^2} \right) \right) \\
    &\geq Mb\left (\frac{n-2}{a}-\frac{(n-2)^2}{2a^2}  \right ) +Ma \left( \frac{-1.1b}{(n+a-2)^2}  \right) \\
    &= Mb \left( \frac{(10n-31)a^3 + 5(n-2)^2(3a^2-(n-2)^2)}{10a^2(a+n-2)^2} \right) \geq 0
    \end{align*}
as desired.  Now, suppose that $n=3$ and $b = s(s+1)$.  Then,
  \begin{align*}
        &s(s+1)\ln \left (1+ \frac{2}{a} \right) + 3s(s+1)\ln \left (1 + \frac{1}{a} \right ) 
    + 3a\ln \left (1 - \frac{s}{a+1} \right) + 3as\ln \left (1 + \frac{1}{a+1} \right) \\
    &\geq b\left (\frac{2}{a} - \frac{2}{a^2} + \frac{3}{a}-\frac{3}{2a^2}  \right ) +3a \left( \frac{-1.1b}{(a+1)^2}  \right) \\
    &= b \left( \frac{a(17a^2-20)+5(13a^2-7)}{10a^2(a+1)^2} \right) \geq 0
    \end{align*}
as desired.
\end{proof}

Finally, to complete our proof of Theorem~\ref{thm: complete} we must prove (\ref{key}) when $n=2$.  In particular, we must prove the following inequality.
\begin{lem} \label{lem: nis2}
 For $a \geq 2$ and $0 \leq s \leq a$,
 \begin{align*} \left (1 + \frac{1}{a} \right )^{s} +  \left [ \left (1 - \frac{s}{a} \right ) \left (1+\frac{1}{a} \right )^s \right ]^{\frac{a}{s+1}}   \geq 2. 
 \end{align*}   
\end{lem}

\begin{proof}
   Throughout the proof, let $r = s/a$, and note that $0 \leq r \leq 1$.  Using $\ln(1+x) \geq x - x^2/2$ for positive real values of $x$, we obtain
  $$ \left (1 + \frac{1}{a} \right )^{s} \geq e^{r - \frac{r}{2a}} \geq e^{0.75r}.$$
  Now, for some real $\tau \in (0.5,\infty)$ notice that if $r$ is such that $\ln(1-r) \geq -r - \tau r^2$, we have
  $$ \left [ \left (1 - \frac{s}{a} \right ) \left (1+\frac{1}{a} \right )^s \right ]^{\frac{a}{s+1}} \geq e^{\frac{a}{s+1} \left(-r - \tau r^2 + r - \frac{r}{2a} \right)} \geq e^{\frac{-\tau}{s+1}(rs+r)} \geq e^{-\tau r}.$$
  So, whenever $r$ is such that $\ln(1-r) \geq -r - \tau r^2$, 
  $$\left (1 + \frac{1}{a} \right )^{s} +  \left [ \left (1 - \frac{s}{a} \right ) \left (1+\frac{1}{a} \right )^s \right ]^{\frac{a}{s+1}} \geq e^{0.75r} + e^{-\tau r}.$$
Finally, one can use basic ideas from calculus to verify the following facts.  When $r \geq 0.93$, $e^{0.75r} \geq 2$.  When $0.78 \leq r \leq 0.93$, $\ln(1-r) \geq -r-2r^2$ and $e^{0.75r}+e^{-2r} \geq 2$. When $0.53 \leq r \leq 0.78$, $\ln(1-r) \geq -r-1.25r^2$ and $e^{0.75r}+e^{-1.25r} \geq 2$. When $0.34 \leq r \leq 0.53$, $\ln(1-r) \geq -r-r^2$ and $e^{0.75r}+e^{-r} \geq 2$.  Finally, when $0 \leq r \leq 0.34$, $\ln(1-r) \geq -r-0.75r^2$ and $e^{0.75r}+e^{-0.75r} \geq 2$.  This completes the proof.
\end{proof}

Having proven (\ref{key}), we are ready to bring all the ingredients together and give a short proof of Theorem~\ref{thm: complete} which we restate.

\begin{customthm}{\bf \ref{thm: complete}}
For each $n, a \in \N$, $\chi_{\ell}(K_n \square K_{a,b}) = n+a$ if and only if $b \geq (n+a-1)!^a/(a-1)!^a$. That is, $f_a(K_n) =   \left( \frac{(n+a-1)!}{(a-1)!} \right)^a$ for each $n, a \in \N$.
\end{customthm}
\begin{proof}
By Theorem~\ref{thm: generalupper}, $f_a(K_n) \leq P_{\ell}(M, n+a-1)^a$. Thus, it suffices to show that $f_a(K_n) \geq P_{\ell}(M, n+a-1)^a$; that is, we want to show that if $M = K_n$ and $H = M \square K_{a, b}$, where $b = P_{\ell}(M, n+a-1)^a - 1$, then $\chi_{\ell}(M\square K_{a, b}) \leq n+a-1$. 

Suppose for the sake of contradiction that there exists an $(n+a-1)$-assignment $L$ for $H$ for which there is no proper $L$-coloring of $H$. By Observation~\ref{obs: disjointchoosable}, we may assume that the lists $L(v_i, x_1), \dots, L(v_i, x_a)$ are pairwise disjoint, for all $i \in [n]$.  By (\ref{key}), 
\[\card{\mathcal{I}_X} + \frac{\card{\mathcal{C}_X} - \card{\mathcal{I}_X}}{2^{n-1}} \geq P_{\ell}(M, n+a-1)^a > b.\]
By Lemma~\ref{lem: gencolorbound}, there is a proper $L$-coloring of $H$ which is a contradiction.
\end{proof}

\end{document}